\documentclass[10pt]{amsart}
\usepackage{amssymb}
\usepackage{epsfig}
\usepackage{url}
\usepackage{setspace}
\theoremstyle{plain}

\newtheorem{thm}{Theorem}[section]
\newtheorem{cor}[thm]{Corollary}
\newtheorem{lem}[thm]{Lemma}
\newtheorem{prop}[thm]{Proposition}

\newtheorem{conj}[thm]{Conjecture}

\newtheorem{prob}[thm]{Problem}
\def\cal{\mathcal}
\def\bbb{\mathbb}
\def\op{\operatorname}
\renewcommand{\phi}{\varphi}

\newcommand{\N}{\bbb{N}}
\newcommand{\Z}{\bbb{Z}}
\newcommand{\Q}{\bbb{Q}}

\begin{document}

\title[Arithmetic properties of Stern polynomials]{Arithmetic properties of the sequence of degrees of Stern polynomials and related results}
\author{Maciej Ulas}


\begin{abstract}
Let $B_{n}(t)$ be a $n$-th Stern polynomial and let
$e(n)=\op{deg}B_{n}(t)$ be its degree. In this note we continue our
study started in \cite{Ul} of the arithmetic properties of the
sequence of Stern polynomials and the sequence
$\{e(n)\}_{n=1}^{\infty}$. We also study the sequence
$d(n)=\op{ord}_{t=0}B_{n}(t)$. Among other things we prove that
$d(n)=\nu(n)$, where $\nu(n)$ is the maximal power of 2 which
dividies the number $n$. We also count the number of the solutions
of the equations $e(m)=i$ and $e(m)-d(m)=i$ in the interval
$[1,2^{n}]$. We also obtain an interesting closed expression for a
certain sum involving Stern polynomials.
\end{abstract}

\maketitle

\section{Introduction}\label{section1}

The {\it Stern sequence} (or {\it Stern's diatomic sequence}) $s(n)$
was introduced in \cite{Ste} and is defined recursively in the
following way
\begin{equation*}
s(0)=0,\quad s(1)=1,\quad s(n)=
\begin{cases}
\begin{array}{lll}
  s(\frac{n}{2})                     & \mbox{if}& n\equiv0\pmod{2},  \\
  s(\frac{n-1}{2})+s(\frac{n+1}{2}) &  \mbox{if}& n\equiv1\pmod{2}.
\end{array}
\end{cases}
\end{equation*}
This sequence appears in different mathematical contexts and was an
object of study of many mathematicians like Lehmer \cite{Leh},
Reznick \cite{Rez} and De Rham \cite{Rham}. A comprehensive survey
of its properties can be found in \cite{Urb}. An interesting survey
of known results and applications of the Stern sequence can also be
found in \cite{Nor}.

In a recent paper \cite{Kla} Klav\v{z}ar, Milutinovi\'{c} and Petr
introduced an interesting polynomial analogue of $s(n)$. More
precisely, they define the sequence $\{B_{n}(t)\}_{n=0}^{\infty}$ of
Stern polynomials as follows: $B_{0}(t)=0, B_{1}(t)=1$ and for
$n\geq 2$ we have
\begin{equation*}
B_{n}(t)=\begin{cases}
\begin{array}{lll}
  tB_{\frac{n}{2}}(t) & \mbox{if} & n\equiv0\pmod{2}, \\
  B_{\frac{n-1}{2}}(t)+B_{\frac{n+1}{2}}(t) & \mbox{if}  & n\equiv1\pmod{2}.
\end{array}
\end{cases}
\end{equation*}

The equality $s(n)=B_{n}(1)$ justifies the name of the sequence
$B_{n}(t)$. In \cite{Kla} it is shown that the sequence of Stern
polynomials has an interesting connection with some combinatorial
objects. In particular the $i$-th coefficient in $B_{n}(t)$ counts
the number of hyperbinary representations of $n-1$ with exactly $i$
occurrences of 1. Moreover, if $e(n)=\op{deg}B_{n}(t)$, then the
number $e(n)$ is equal to the difference between the length and the
weight of the non-adjacent form of $n$. These two properties show
that the polynomials $B_{n}(t)$ are an interesting object of study.
We investigated these polynomials and associated sequence $e(n)$ in
a recent paper \cite{Ul}. In this note we continue our study. We
give now a short introduction about the content of the paper.

In Section \ref{section2} we introduce the sequence
\begin{equation*}
d(n)=\op{ord}_{t=0}B_{n}(t),
\end{equation*}
and study its relations with the sequence $\{e(n)\}_{n=1}^{\infty}$.
We show that $d(n)$ is just $\mu(n)$, the maximal power of 2 which
divides $n$. Among other things we also count the number $e(i,n)$ of
solutions of the equation $e(m)-d(m)=i$ in the interval $[1,2^{n}]$.

In Section \ref{section3} we study the sum
\begin{equation*}
S_{k}(n)=\sum_{i:\;e(i)=n}i^{k},
\end{equation*}
where $k, n$ are given. In particular we give recurrence relations
satisfied by the sequence
$G_{k}(x)=\sum_{n=0}^{\infty}S_{k}(n)x^{n}$, $k=0,1,2,\ldots$.

In Section \ref{section4} we give an alternative definition of the
polynomial $B_{n}(t)$ as a determinant of a certain matrix.

In Section \ref{section5} we give a generalization of a certain sum
given in Urbiha's paper \cite{Urb}. Finally, in the last section we
give some additional results on the sequence
$\{e(n)\}_{n=1}^{\infty}$ and state some open problems and
conjectures which appear during our investigations and which we were
unable to prove.

\section{Relations between $d(n)$ and $e(n)$}\label{section2}

We define the following sequences
\begin{equation*}
d(n)=\op{ord}_{t=0}B_{n}(t),\quad e(n)=\op{deg}B_{n}(t).
\end{equation*}

The sequence $e(n), n=1,2,\ldots$ was introduced in the paper
\cite{Kla}. We have $e(1)=0$, $e(2)=1$ and for $n\geq 3$:
\begin{equation*}
e(2n)=e(n)+1,\quad e(2n+1)=\op{max}\{e(n),\;e(n+1)\}
\end{equation*}
An alternative recurrence relation which is more convenient was
obtained in \cite[Corollary 13]{Kla} and has the form
\begin{equation*}
e(2n)=e(n)+1,\quad e(4n+1)=e(n)+1,\quad e(4n+3)=e(n+1)+1.
\end{equation*}
We will use it several times in the sequel. Additional arithmetic
properties which will be useful in our investigations were obtained
in \cite{Ul}. Because we will it use several times we recall it
without proof.

\begin{thm}[Theorem 4.3 in \cite{Ul}]\label{citedthm}
 We have the following equalities:
\begin{equation*}
\begin{array}{lcl}
  m(n) & = & \op{min}\{e(i):\;i\in
[2^{n-1},2^{n}]\}=\left\lfloor\frac{n}{2}\right\rfloor,\;n\geq 2, \\
  M(n) & = & \op{max}\{e(i):\;i\in [2^{n-1},2^{n}]\}=n.
  \end{array}
  \end{equation*}
Moreover,
\begin{equation*}
\begin{array}{lcl}
  \op{mdeg}(n) & = & \op{min}\{i:\;e(i)=n\}=2^{n}, \\
  \op{Mdeg}(n) & = & \op{max}\{i:\;e(i)=n\}=\frac{4^{n+1}-1}{3}.
\end{array}
\end{equation*}
\end{thm}

\begin{thm}
We have the equality $d(n)=\nu(n)$, where
$\nu(n)=\op{max}\{k:\;2^{k}\;\mbox{divide}\;n\}$.
\end{thm}
\begin{proof}
First of all let us note that $d(1)=0$, $d(2)=1$, $d(3)=0$ and
$d(4)=2$. We show that the sequence $d(n)$ satisfies the following
relations $d(2n)=d(n)+1$ and $d(2n+1)=0$. These relations clearly
hold for $n\leq 4$. From the definition of the sequence $d(n)$ as
$\op{ord}_{t=0}B_{n}(t)$ we deduce the following relations
\begin{equation*}
d(2n)=d(n)+1,\quad\quad d(2n+1)=\op{min}\{d(n),d(n+1)\}.
\end{equation*}
Now let us note that
\begin{equation*}
d(4n+1)=\op{min}\{d(2n),d(2n+1)\}=\op{min}\{d(n)+1,\op{min}\{d(n),d(n+1)\}\}
\end{equation*}
which shows that $d(4n+1)=\op{min}\{d(n),d(n+1)\}=d(2n+1)$.
Similarly, we get that
\begin{equation*}
d(4n+3)=\op{min}\{d(2n+1),d(2n+2)\}=\op{min}\{\op{min}\{d(n),d(n+1)\},d(n+1)+1\}
\end{equation*}
and thus we get the equality $d(4n+3)=d(2n+1)$. Because $d(1)=0$ by
induction on $n$ we get that $d(2n+1)=0$ for all $n$. This
conclusion finishes the proof of our theorem due to the fact that
the sequence $\nu(n)$ satisfies exactly the same recurrence relation
as $d(n)$ and we have equality $d(n)=\nu(n)$ for $n=1,2,3,4$. Thus
we deduce that $d(n)=\nu(n)$ for all $n$.
\end{proof}

It is clear that we have an inequality $d(n)\leq e(n)$ for all $n$
and thus we can define the map
\begin{equation*}
\Phi:\;\N_{+}\ni n\mapsto (d(n),e(n))\in \{(a,b)\in
\N\times\N:\;a\leq b\}.
\end{equation*}

We have the following.

\begin{prop}
The map $\Phi$ is onto.
\end{prop}
\begin{proof}
This is very simple. We show that for any pair of nonnegative
integers $(p, q)$ with $p\leq q$ there exists a natural number $n$
such that $\nu(n)=p$ and $e(n)=q$. If $p=q$ it is enough to take
$n=2^{p}$. We can assume that $p<q$. In order to prove the demanded
property we take $n=2^{p}(2^{q-p+1}+1)$. We clearly have $\nu(n)=p$.
In order to finish the proof it is enough to show that
$e(2^{m+1}+1)=m$ for positive $m$. We prove this by induction on
$m$. The equality is clearly true for $m=1$ because $e(5)=1$. Let us
suppose that the desired equality is true for each $k\leq m$. We
compute
\begin{equation*}
e(2^{m+2}+1)=e(4\cdot2^{m}+1)=e(2^{m})+1=m+1,
\end{equation*}
and get the desired result. This observation finishes the proof of
the surjectivity of the map $\Phi$.
\end{proof}
\bigskip

It is an interesting question what can be said about the solutions
of the equation $e(n)-d(n)=i$, where $i\in\N$ is given. More
precisely we are interested in the problem of counting the number of
elements of the set $C(i,n),$ where
\begin{equation*}
C(i,n)=\{m\in [1,2^{n}]:\;e(m)-d(m)=i\}.
\end{equation*}
We define $c(i,n)=|C(i,n)|$ and note that $c(i,n)$ exists for $i\leq
n$ which follows from the properties of the sequence $e(n)$
presented in Theorem \ref{citedthm}. Now let us note that
\begin{align*}
c(i,n+1)&=|\{m\in [1,2^{n+1}]:\;e(m)-d(m)=i\}|\\
        &=|\{m\in [1,2^{n}]:\;e(2m)-d(2m)=i\}|\\
        &\quad+|\{m\in [0,2^{n}-1]:\;e(2m+1)-d(2m+1)=i\}|\\
        &=c(i-1,n)+|\{m\in [0,2^{n}-1]:\;e(2m+1)=i\}|,
\end{align*}
where in the last equality we use the fact that $d(2m+1)=\nu(2m+1)=0$.

We thus see that in order to compute the $c(i,n)$ we need to know
the value of $|\{m\in [0,2^{n}-1]:\;e(2m+1)=i\}|$. In order to do
this we will need some properties of the polynomial
\begin{equation*}
H_{n}(x)=\sum_{i=1}^{2^{n}}x^{e(i)}=\sum_{i=0}^{n}e(i,n)x^{i},
\end{equation*}
where the equality $\op{deg}H_{n}=n$ follows from the properties of
the sequence $e(n)$, and the number $e(i,n)$ is the cardinality of
the set $\{m\in [1,2^{n}]:\;e(m)=i\}$.

\begin{lem}
We have $H_{0}(x)=1$, $H_{1}(x)=x+1$ and for $n\geq 2$ we get that
$H_{n}(x)$ satisfies the recurrence relation
\begin{equation*}
H_{n+2}(x)=xH_{n+1}(x)+2xH_{n}(x)-x^{n+1}+1.
\end{equation*}
Moreover, we have the equality $e(0,n)=1$ and for $1\leq i\leq n$ we
have
\begin{equation*}
e(i,n+2)=e(i-1,n+1)+2e(i-1,n)-[i=n+1],
\end{equation*}
where as usual $[A]$ is equal to {\rm 1} if $A$ is true and {\rm 0}
otherwise.
\end{lem}

\begin{proof}
We clearly have that $H_{0}(x)=1$ and $H_{1}(x)=x+1$. Let us assume
that $n\geq 2$. Then we have the following chain of equalities
\begin{align*}
H_{n+2}(x)&=\sum_{i=1}^{2^{n+2}}x^{e(i)}=\sum_{i=1}^{2^{n+1}}x^{e(2i)}+1+\sum_{i=1}^{2^{n}}x^{e(4i+1)}-x^{e(2^{n+2}+1)}+\sum_{i=0}^{2^{n}-1}x^{e(4i+3)}\\
          &=xH_{n+1}(x)+1+\sum_{i=1}^{2^{n}}x^{e(i)+1}-x^{n+1}+\sum_{i=0}^{2^{n}-1}x^{e(i+1)+1}\\
          &=xH_{n+1}(x)+1+2xH_{n}(x)-x^{n+1}.
\end{align*}
This proves the first part of our proposition. Comparing now the
coefficients on both sides of the obtained equality we get the
second part of the proposition.
\end{proof}

As an immediate consequence of the above lemma we get the following.

\begin{cor}
We have $\sum_{i=0}^{2^{n}-1}x^{e(2i+1)}=H_{n+1}(x)-xH_{n}(x)$ and
for $i,n\in\N$ with $i\leq n$ we get
\begin{equation*}
|\{m\in [0,2^{n}-1]:\;e(2m+1)=i\}|=e(i,n+1)-e(i-1,n).
\end{equation*}
\end{cor}

From the above corollary we immediately deduce the recurrence
relation for $c(i,n)$. We have
$c(i,n)=c(i-1,n-1)+e(i,n)-e(i-1,n-1)$, which can be rewritten as
\begin{align*}
c(i,n)&-e(i,n)=c(i-1,n-1)-e(i-1,n-1)\\
      &=c(i-2,n-2)-e(i-2,n-2)=\ldots=c(0,n-i)-e(0,n-i)=0.
\end{align*}
We thus deduce that $c(i,n)=e(i,n)$ and we left with the problem of
computation of the coefficients of the polynomial $H_{n}(x)$. In
order to do this we start with the following.

\begin{lem}\label{genforH(n)}
Let $n\geq 0$ and consider the polynomial
$H_{n}(x)=\sum_{i=1}^{2^{n}}x^{e(i)}$. Then, we have an identity
\begin{equation*}
\cal{E}(x,y)=\sum_{n=0}^{\infty}H_{n}(x)y^{n}=\frac{1 -
xy(1+y-y^2)}{(1-y)(1-xy)(1-xy-2xy^2)}.
\end{equation*}
In particular we have
\begin{equation*}
H_{n}(x)=\frac{1}{1-3x}+\frac{1}{2}x^{n}+h_{n}(x),
\end{equation*}
where
\begin{equation*}
h_{n}(x)=\frac{(\sqrt{2x}i)^{n}}{2(1-3x)}\left(8T_{n}\left(-\frac{\sqrt{2x}i}{4}\right)-3(x+3)U_{n}\left(-\frac{\sqrt{2x}i}{4}\right)\right)
\end{equation*}
and $T_{n}(x)$ (respectively $U_{n}(x)$) is the Chebyshev polynomial
of the first kind (respectively of the second kind) and $i^2=-1$.
\end{lem}
\begin{proof}
In order to prove the demanded equality we use the recurrence
relation for $H_{n}(x)$. We have
\begin{align*}
\cal{E}(x,y)=1&+(1+x)y+\sum_{n=0}^{\infty}H_{n+2}(x)y^{n+2}=1+(1+x)y+xy\sum_{n=0}^{\infty}H_{n+1}(x)y^{n+1}\\
              &+2xy^2\sum_{n=0}^{\infty}H_{n}(x)y^{n}-y\sum_{n=0}^{\infty}x^{n+1}y^{n+1}+\sum_{n=0}^{\infty}y^{n+2}\\
              &=1+(1+x)y+xy(\cal{E}(x,y)-1)+2xy^2\cal{E}(x,y)-\frac{xy^2}{1-xy}+\frac{y^2}{1-y}.
\end{align*}
Solving now the obtained (linear) equation with respect to
$\cal{E}(x,y)$ we get the expression from the statement of the
corollary. In order to get the expression for $H_{n}(x)$ we use the
standard method of decomposition of rational into simple fractions.
More precisely we have
\begin{equation*}
\cal{E}(x,y)=\frac{1}{1-3x}\cdot\frac{1}{1-y}+\frac{1}{2(1-xy)}-\frac{1+3x+4xy}{2(1-3x)(1-xy-2xy^2)}.
\end{equation*}
Let us denote the last term by $F(x,y)$. We express the $n$-th
coefficient, say $h_{n}(x)$, of the power series expansion of the
function $F(x,y)$ with respect to $y$ as a certain combination of
Chebyshev polynomials of the first and the second kind in the
variable $-\sqrt{2x}i/4$. Let us recall that the $n$-th Chebyshev
polynomial of the first kind is defined by the relation
$T_{n}(x)=\cos (n\arccos x)$. The $n$-th Chebyshev polynomial of the
second kind is defined as $U_{n}(x)=\sin((n+1)\arccos
x)/\sin(\arccos x)$. The generating function $T(x,y)$ (respectively
$U(x,y)$) for the sequence of Chebyshev polynomials of the first
kind (respectively the second kind) is given by
\begin{equation*}
T(x,y)=\frac{1-xy}{1-2xy+y^2},\quad\quad U(x,y)=\frac{1}{1-2xy+y^2}.
\end{equation*}

We express now the function $F(x,y)$ as a combination of the
functions $T(x,y)$ and $U(x,y)$. More precisely we have the
following equality
\begin{equation*}
F(x,y)=\frac{1}{2(1-3x)}\left(8T\left(-\frac{\sqrt{2x}i}{4},\sqrt{2x}iy\right)-3(x+3)U\left(-\frac{\sqrt{2x}i}{4},\sqrt{2x}iy\right)\right).
\end{equation*}
Comparing now the coefficients on the both sides of the above
equality we get the following expression for the rational function
$h_{n}(x)$:
\begin{equation*}
h_{n}(x)=\frac{(\sqrt{2x}i)^{n}}{2(1-3x)}\left(8T_{n}\left(-\frac{\sqrt{2x}i}{4}\right)-3(x+3)U_{n}\left(-\frac{\sqrt{2x}i}{4}\right).
  \right)
\end{equation*}
This is exactly the expression for the function $h_{n}(x)$ displayed
in the statement of the theorem.
\end{proof}

In order to give a closed value of the number $e(i,n)$ we use the
expression for the polynomial $H_{n}(x)$. However, before we do that
let us recall how the coefficients of the polynomials $T_{n}(x)$ and
$U_{n}(x)$ look like:
\begin{align*}
T_{n}(x)&=\frac{n}{2}\sum_{k=0}^{\left[\frac{n}{2}\right]}\frac{(-1)^{k}}{n-k}C(n-k,k)(2x)^{n-2k},\\
U_{n}(x)&=\sum_{k=0}^{\left[\frac{n}{2}\right]}(-1)^{k}C(n-k,k)(2x)^{n-2k},
\end{align*}
where as usual $C(a,b)=\left(\begin{array}{c}
 a\\
  b\\
\end{array}\right)$ is a binomial coefficient. All properties of
the Chebyshev polynomials of both kinds which we have used can be
found in \cite{Riv}.

In view of the above identities for $U_{n}(t)$ and $T_{n}(t)$ we get
the expression for $H_{n}(x)$ in the following form
\begin{align*}
H_{n}(x)&=\frac{1}{2}x^{n}+\frac{1}{1-3x}\left(1+\sum_{k=0}^{\left[\frac{n}{2}\right]}\frac{(5k-3n)2^{k}}{n-k}C(n-k,k)x^{n-k}\right)\\
        &=\frac{1}{2}x^{n}+\frac{1}{1-3x}\left(1+\sum_{i=n-\lfloor\frac{n}{2}\rfloor}^{n}\frac{(2n-5i)2^{n-i}}{i}C(i,n-i)x^{i}\right).
\end{align*}
Let us put $Z_{n}(x)=(1-3x)(H_{n}(x)-\frac{1}{2}x^{n})$. We thus see
that in order to find a formula for $e(i,n)$ it is enough to compute
the coefficients of the polynomial $Z_{n}(x)/(1-3x)$. Now, we recall
that
\begin{equation*}
\sum_{i=0}^{n}a_{i,n}x^{i}=(1-cx)\sum_{i=0}^{n-1}b_{i,n-i}\Longleftrightarrow
b_{i,n-1}=c^{i}\sum_{j=0}^{i}\frac{a_{j,n}}{c^{j}}.
\end{equation*}
In our situation $c=3$ and we have the following expression for the
coefficients of the polynomial $Z_{n}(x)$:
\begin{equation*}
a_{i,n}=\begin{cases}
\begin{array}{lcl}
  1 & \mbox{for} & i=0, \\
  0 & \mbox{for} & 1\leq i< n-\lfloor\frac{n}{2}\rfloor,\\
  \frac{(2n-5i)2^{n-i}}{i}C(i,n-i) & \mbox{for} & n-\lfloor\frac{n}{2}\rfloor\leq i\leq
  n.
\end{array}
\end{cases}
\end{equation*}
Computing now the expression for $b_{i,n-1}$ using the above
equality and performing all necessary simplifications we get the
following result.

\begin{thm}
Let $e(i,n)=|\{m\in [1,2^{n}]:\;e(m)=i\}|$. Then we have
$H_{n}(x)=\sum_{i=1}^{n}e(i,n)x^{i}$, where
\begin{equation*}
e(i,n)=\begin{cases}
\begin{array}{lcl}
  3^{i} & \mbox{for} & 0\leq i \leq n-\lfloor\frac{n}{2}\rfloor-1, \\
  \frac{2^{n}C(i,n-i)}{2^{i+1}}+3^{i}\left(1+2^{n}\sum_{j=n-\lfloor\frac{n}{2}\rfloor}^{i}\frac{2n-5j}{j6^{j}}C(j,n-j)\right) & \mbox{for} & n-\lfloor\frac{n}{2}\rfloor\leq i\leq n-1,  \\
  1 & \mbox{for} & i=n.
\end{array}
\end{cases}
\end{equation*}
\end{thm}

We tried to obtain more a convenient expression for the numbers
$e(i,n)$ with $n-\lfloor\frac{n}{2}\rfloor\leq i\leq n-1$. However
we have been unable to do this. This leads to the following

\begin{prob}
Find a simpler expression for the number $e(i,n)$ in case of
$n-\lfloor\frac{n}{2}\rfloor\leq i\leq n-1$.
\end{prob}

\section{Generating function for the sum $S_{k}(n)=\sum_{a:\;e(a)=n}a^{k}$ }\label{section3}

In this section we prove one more interesting result concerning the
sequence $\{e(n)\}_{n=1}^{\infty}$. More precisely, for a fixed
integer $n$ we are interested in the computation of a sum involving
all elements of the form $i^{k}$, where $i$ is a solution of the
equation $e(i)=n$ and $k$ is given. In the recent paper \cite{Ul} we
compute this sum for $k=0$, i. e. we compute the number of solutions
of the equation $e(i)=n$. Let us introduce the set
\begin{equation*}
\cal{M}(n)=\{i\in\N:\;e(i)=n\}.
\end{equation*}
From the cited result we know that the cardinality of the set
$\cal{M}(n)$ is $3^{n}$. We thus see that the sum we are interested
in
\begin{equation*}
S_{k}(n)=\sum_{a\in\cal{M}(n)}a^{k},
\end{equation*}
is finite for any given $k$ and $n$. In order to tackle the problem
we introduce the generating function
\begin{equation*}
G_{k}(x)=\sum_{n=1}^{\infty}n^{k}x^{e(n)}=\sum_{n=0}^{\infty}S_{k}(n)x^{n}.
\end{equation*}
We know that $G_{0}(x)=1/(1-3x)$. We show how $G_{k}(x)$ can be
computed. We prove the following.

\begin{thm}\label{sumofi}
Let $k$ ba a positive integer and let us consider the function
$G_{k}(x)$. Then $G_{k}(x)$ can be computed as a linear combination
of $G_{i}(x), (i=0,1,\ldots,k-1)$ with rational functions as
coefficients. More precisely we have the following expression:
\begin{equation*}
G_{k}(x)=\frac{x\sum_{j=0}^{k-1}C(k,j)(4^{j}+(-1)^{k}(-4)^{j})G_{j}(x)+1}{1-2^{k}(2^{k+1}+1)x}.
\end{equation*}
\end{thm}
\begin{proof}
Using the recurrence relation for the sequence $e(n)$ and the
expression for $G_{0}(x)$ we obtain the recurrence relation
satisfied by the sequence $G_{k}(x),\;k=0,1,2,\ldots .$ We have the
following equality:
\begin{equation*}
G_{k}(x)=\sum_{i=1}^{\infty}i^{k}x^{e(i)}=\sum_{i=1}^{\infty}(2i)^{k}x^{e(2i)}+1+\sum_{i=1}^{\infty}(4i+1)^{k}x^{e(4i+1)}+\sum_{i=0}^{\infty}(4i+3)^{k}x^{e(4i+3)}
\end{equation*}
Now, let us note that
$\sum_{i=1}^{\infty}(2i)^{k}x^{e(2i)}=2^{k}xG_{k}(x)$. Moreover, we
have an equality
\begin{equation*}
\sum_{i=1}^{\infty}(4i+1)^{k}x^{e(4i+1)}=x\sum_{i=1}^{\infty}\sum_{j=0}^{k}C(k,j)4^{j}i^{j}x^{e(i)}=x\sum_{j=0}^{k}C(k,j)4^{j}G_{j}(x),
\end{equation*}
and
\begin{align*}
\sum_{i=0}^{\infty}&(4i+3)^{k}x^{e(4i+3)}=x\sum_{i=0}^{\infty}\sum_{j=0}^{k}(4(i+1)-1)^{k}x^{e(i+1)}=x\sum_{i=1}^{\infty}\sum_{j=0}^{k}(4i-1)^{k}x^{e(i)}\\
                   &=(-1)^{k}x\sum_{i=1}\sum_{j=0}^{k}C(k,j)(-4)^{j}i^{j}x^{e(i)}=(-1)^{k}x\sum_{j=0}^{k}C(k,j)(-4)^{j}G_{j}(x).
\end{align*}
Gathering now the obtained expressions together we get
\begin{align*}
G_{k}(x)&=2^{k}xG_{k}(x)+x\sum_{j=0}^{k}C(k,j)4^{j}G_{j}(x)+(-1)^{k}x\sum_{j=0}^{k}C(k,j)(-4)^{j}G_{j}(x)+1\\
        &=2^{k}(2^{k+1}+1)xG_{k}(x)+x\sum_{j=0}^{k-1}C(k,j)(4^{j}+(-1)^{k}(-4)^{j})G_{j}(x)+1.
\end{align*}

Solving now the above functional equation with respect to $G_{k}(x)$
(this is a linear equation!) we get the expression displayed in the
statement of the theorem.
\end{proof}

From the theorem we have just proved we deduce the following set of
corollaries.

\begin{cor}
Let $k$ be a nonnegative integer and let us consider the generating
function $G_{k}(x)$.
\begin{enumerate}
\item If $k$ is even then $G_{k}(x)$ is a linear combination (with rational functions as
coefficients) of the functions $G_{2i}(x)$ for
$i=0,1,\ldots,\frac{k}{2}$.

\item If $k$ is odd then $G_{k}(x)$ is a linear combination (with rational functions as
coefficients) of the functions $G_{2i-1}(x)$ for
$i=0,1,\ldots,\frac{k+1}{2}$.
\end{enumerate}
\end{cor}
\begin{proof}
This is an immediate consequence of the expression displayed in the
statement of Theorem \ref{sumofi}. Indeed, if $k$ is even then the
expression for $G_{k}(x)$ contains only terms with $G_{2i}$ for
$i=0,1,\ldots,k/2$. Similar reasoning applies in the case of $k$
odd.
\end{proof}

\begin{cor}
Let $k, n\in\N$ and let us put $T_{i}=2^{i}(2^{i+1}+1)$. Then there
exist rational numbers $\alpha_{i}$ for $i=0,1,\ldots k$ such that
\begin{equation*}
S_{k}(n)=\begin{cases}\begin{array}{lll}
                        \sum_{i=0}^{\frac{k}{2}}\alpha_{2i}T_{2i}^{n}       & \mbox{if} & k\;\mbox{is even}, \\
                        \\
                        \sum_{i=1}^{\frac{k+1}{2}}\alpha_{2i-1}T_{2i-1}^{n} & \mbox{if} & k\;\mbox{is
                        odd}.
                      \end{array}
\end{cases}
\end{equation*}
\end{cor}

\begin{cor}
Let $n\geq 0$ be given. Then we have the following equalities
\begin{equation*}
S_{1}(n)=10^{n},\quad \quad S_{2}(n)=\frac{35\cdot 36^{n}-2\cdot
3^{n}}{33},\quad\quad S_{3}(n)=\frac{25\cdot 136^{n}-4\cdot
10^{n}}{21}.
\end{equation*}
\end{cor}
\begin{proof}
We use the result from the previous theorem. If $k=1$ then we easily
get that $G_{1}(x)=\frac{1}{1-10x}$. Comparing now the coefficients
of the power series expansions of the two functions we get the
expression for $S_{1}(n)$ displayed in the statement. For $k=2$ we
get that
\begin{equation*}
G_{2}(x)=\frac{1-x}{(1-3x)(1-36x)}=\frac{35}{33}\cdot\frac{1}{1-36x}-\frac{2}{33}\cdot\frac{1}{1-3x}.
\end{equation*}
Using the same method as in the case of $G_{1}(x)$ we get the
expression for $S_{2}(n)$ displayed in the statement of the
corollary.

Finally, if $k=3$ then
\begin{equation*}
G_{3}(x)=\frac{1 +
14x}{(1-10x)(1-136x)}=\frac{25}{21}\cdot\frac{1}{1-136x}-\frac{4}{21}\cdot\frac{1}{1-10x},
\end{equation*}
and we easily get the expression for $S_{3}(n)$.
\end{proof}

\section{A determinant expression for the $B_{n}(t)$}\label{section4}

In a recent paper \cite[Theorem 2.4]{Ul} we proved that if $\mu(n)$
is the highest power of $2$ dividing $n$ then the following identity
holds
\begin{equation*}
t^{\mu(n)}(B_{n+1}(t)+B_{n-1}(t))=(B_{2^{\mu(n)}+1}(t)+B_{2^{\mu(n)}-1}(t))B_{n}(t).
\end{equation*}
From the above identity we easily deduce that the sequence
$\{B_{n}(t)\}_{n=0}^{\infty}$ satisfies a three term recurrence
relation with variable coefficients of the form
\begin{equation}\label{altrec}
B_{0}(t)=0,\quad B_{1}(t)=1,\quad
B_{n+1}(t)=A_{n}(t)B_{n}(t)-B_{n-1}(t),
\end{equation}
where
\begin{equation*}
A_{n}(t)=t^{-\mu(n)}(B_{2^{\mu(n)}+1}(t)+B_{2^{\mu(n)}-1}(t)).
\end{equation*}
Because $B_{2^{n}-1}(t)=(t^{n}-1)/(t-1)$ and
$B_{2^{n}+1}(t)=(t^{n}-1)/(t-1)+t$ \cite[Corollary 2.2]{Ul}, we get
that
\begin{equation*}
A_{n}(t)=t^{-\mu(n)}\left(2\frac{t^{\mu(n)}-1}{t-1}+t\right).
\end{equation*}
We thus see that the function $A_{n}(t)$ does not depend on the
polynomials $B_{i}(t),\;i=0,1,\ldots, n$.

The recurrence relation (\ref{altrec}) permits us to give a new
definition of the Stern polynomial $B_{n}(t)$ as a determinant.
Indeed, let us consider (\ref{altrec}), with $n$ replaced by $n-1$,
as a homogeneous linear equation in three "unknowns" $B_{n},
B_{n-1}, B_{n-2}$, namely $B_{n}-A_{n-1}B_{n-1}+B_{n-2}=0$. Next, we
replace $n$ successively by $n-1, n-2, \ldots, 3, 2.$ For $n=3$ the
equation reads $B_{3}-A_{2}B_{2}=-1$. For $n=2$ we have
$B_{2}-A_{1}B_{1}=0$ and finally for $n=1$ we get $B_{1}=1$. We
solve this system of $n$ linear equations in $B_{n}, B_{n-1},\ldots,
B_{1}$ by Cramer's rule. The determinant of the system is
\begin{equation*}
D=\begin{array}{|cccccccc|}
    1      & -A_{n-1}(t)  &         1 & 0        &  \ldots & 0 & 0 &0\\
    0      &        1  &  -A_{n-2}(t) & 1        &  \ldots & 0 & 0 &0\\
        \vdots & \vdots    & \vdots     & \vdots &       &  \vdots & \vdots & \vdots\\
    0 & 0 & 0 & 0 & \ldots &1 & -A_{2}(t) & 0 \\
    0 & 0 & 0 & 0 & \ldots &0 & 1 & -A_{1}(t) \\
    0 & 0 & 0 & 0 & \ldots&0 & 0 & 1
  \end{array}\;,
\end{equation*}

\noindent and $B_{n}(t)D=D_{n}$, where $D_{n}$ is obtained by
replacing the first column of $D$ by the column vector of the
constant terms. The entries of this vector are, as seen, all zeros,
except for the last three, which are $-1$ (from $B_{3}$), $0$ (from
$B_{2}$) and $1$ (from $B_{1}$). Because the matrix associated with
$D$ is triangular we have that $D=1$ and finally we get the
determinant expression for the $n$-th Stern polynomial in the form
\begin{equation*}
B_{n}(t)=\begin{array}{|cccccccc|}
    0      & -A_{n-1}(t)  &         1 & 0        &  \ldots & 0 & 0 &0\\
    0      &        1  &  -A_{n-2}(t) & 1        &  \ldots & 0 & 0 &0\\
        \vdots & \vdots    & \vdots     & \vdots &       &  \vdots & \vdots & \vdots\\
    -1 & 0 & 0 & 0 & \ldots &1 & -A_{2}(t) & 0 \\
    0 & 0 & 0 & 0 & \ldots &0 & 1 & -A_{1}(t) \\
    1 & 0 & 0 & 0 & \ldots&0 & 0 & 1
  \end{array}\;.
\end{equation*}

\section{An interesting sum involving Stern polynomials}\label{section5}

In this section we are interested in the computation of the certain
sum involving the reciprocals of products of two consecutive Stern
polynomials. This is a variation on a theme of generalization of the
(not so well known) identity
\begin{equation*}
\sum_{i=m}^{2m-1}\frac{1}{s(i)s(i+1)}=1,
\end{equation*}
where $s(i)=B_{i}(1)$ is the $i$-th term of the Stern diatomic
sequence. The proof of this identity with an interesting discussion
can be found in \cite{Urb}.

Before we state the main result in this section we introduce an
interesting family of polynomials which will we need in the sequel.
Let $S_{1}(t)=S_{2}(t)=0$ and for $k\geq 1$ we put
\begin{equation*}
S_{2k}(t)=tS_{k}(t),\quad\quad
S_{2k+1}(t)=S_{k}(t)+S_{k+1}(t)+t^{\lfloor\log k \rfloor}.
\end{equation*}
The first few terms of the sequence $\{S_{k}(t)\}_{n=0}^{\infty}$
are contained in the table below
\begin{equation*}
\begin{tabular}{|l|l|l|l|l|l|l|l|}
  $n$ & $S_{n}(t)$ & $n$ & $S_{n}(t)$ & $n$  & $S_{n}(t)$ & $n$ & $S_{n}(t)$ \\
  \hline
  1 &        0 & 5 &      $1+t$ & 9  & $1+t+t^2$          & 13 & $(1+t)^2$\\
  2 &        0 & 6 &        $t$ & 10 & $t(1+t)$           & 14 & $t(1+t)$\\
  3 &        1 & 7 &      $1+t$ & 11 & $(1+t)^2$          & 15 & $1+t+t^2$\\
  4 &        0 & 8 &        0 & 12 & $t^2$                & 16 & 0\\
\end{tabular}
\end{equation*}

In the sequel we will need the following result which connects the
sequences $\{B_{n}(t)\}_{n=0}^{\infty}$ and
$\{S_{n}(t)\}_{n=1}^{\infty}$.

\begin{lem}\label{mainlemma}
Let $n\in\N_{+}$. Then the following identity holds
\begin{equation}\label{impid}
(2-t)(B_{n}(t)S_{n+1}(t)-S_{n}(t)B_{n+1}(t))=t^{\lfloor \log
_{2}k\rfloor}(B_{n+1}(t)-B_{n}(t)-t+1).
\end{equation}
\end{lem}
\begin{proof}
In order to prove our theorem we proceed by induction on $n$. The
desired identity is clearly true for $n=1$ (in this case both sides
are equal to 0) and for $n=2$ (in this case both sides are equal to
$(2-t)t)$. Let us suppose that our identity is true for all numbers
$<n$. We prove that it is true for $n$. We consider two cases: $n$
even and $n$ odd.

If $n$ is even, then $n=2m$ for some $m\in\N_{+}$. In order to
shorten the notation we put $B_{n}:=B_{n}(t)$ and $S_{n}:=S_{n}(t)$.
Recall now that $B_{2m}=tB_{m}$ and $B_{2m+1}=B_{m}+B_{m+1}$. Using
now the recurrence relation for $B_{n}$ and $S_{n}$ we have the
following chain of equalities
\begin{align*}
(2&-t)(B_{2m}S_{2m+1}-S_{2m}B_{2m+1})\\
  &=t(2-t)(B_{m}(S_{m}+S_{m+1}+t^{\lfloor\log_{2}m\rfloor})-S_{m}(B_{m}+B_{m+1}))\\
  &=t(2-t)(B_{m}S_{m+1}-S_{m}B_{m+1})+t^{\lfloor\log_{2}m\rfloor+1}(2-t)B_{m}\\
  &=t^{\lfloor\log_{2}m\rfloor+1}(B_{m+1}-B_{m}-t+1)+t^{\lfloor\log_{2}m\rfloor+1}(2-t)B_{m}\\
  &=t^{\lfloor\log_{2}(2m)\rfloor}(B_{m+1}+B_{m}-tB_{m}-t+1)\\
  &=t^{\lfloor\log_{2}(2m)\rfloor}(B_{2m+1}-B_{2m}-t+1).
\end{align*}
This finishes the proof of our identity in case of $n$ even.

If $n$ is odd, then $n=2m+1$ for some $m\in\N$. Using similar
reasoning as in the previous case we get the following chain of
equalities
\begin{align*}
(2&-t)(B_{2m+1}S_{2m+2}-S_{2m+1}B_{2m+2})\\
  &=t(2-t)((B_{m}+B_{m+1})S_{m+1}-(S_{m}+S_{m+1}+t^{\lfloor\log_{2}m\rfloor})B_{m+1})\\
  &=t(2-t)(B_{m}S_{m+1}-S_{m+1}B_{m})-(2-t)t^{\lfloor\log_{2}m\rfloor+1}B_{m+1}\\
  &=t^{\lfloor\log_{2}m\rfloor+1}(B_{m+1}-B_{m}-t+1)-(2-t)t^{\lfloor\log_{2}m\rfloor+1}B_{m+1}\\
  &=t^{\lfloor\log_{2}m\rfloor+1}(B_{m+1}-B_{m}-t+1-2B_{m+1}+tB_{m+1})\\
  &=t^{\lfloor\log_{2}(2m+1)\rfloor}(B_{2m+2}-B_{2m+1}-t+1),
\end{align*}
where in the last identity we use an obvious fact that $\lfloor
\log_{2}m\rfloor+1=\lfloor \log_{2}(2m+1)\rfloor$. This finishes the
proof of our identity in case of $n$ odd. Gathering now what we have
we deduce that the identity (\ref{impid}) is true for all $n\in\N$.
\end{proof}

From the above lemma we deduce an interesting corollary which will
be important in the proof of the main result of this section.

\begin{cor}\label{mainid}
Let $k\in\N_{+}$. Then we have the following identity
\begin{equation}\label{impid2}
\frac{2-t}{t^{\lfloor
\log_{2}k\rfloor}}\frac{S_{2k+1}(t)}{B_{2k+1}(t)}=\frac{2-t}{t^{\lfloor
\log_{2}k\rfloor}}\frac{S_{k}(t)}{B_{k}(t)}+\frac{B_{k+1}(t)-(t-1)(B_{k}(t)+1)}{B_{k}(t)B_{2k+1}(t)}.
\end{equation}
\end{cor}
\begin{proof}
This is a simple consequence of the identity from Lemma
\ref{mainlemma}. Indeed, we take $n=2k$ and note that
$B_{2k}(t)=B_{k}(t)$ and $S_{2k}(t)=tS_{k}(t)$. Moreover, the right
side of the identity (\ref{impid}) is clearly equal to
$B_{k+1}(t)-(t-1)(B_{k}(t)+1)$. Dividing now both sides of
(\ref{impid}) by $t^{\lfloor\log_{2}m\rfloor}B_{k}(t)B_{2k+1}(t)$
and adding the expression $\frac{2-t}{t^{\lfloor
\log_{2}k\rfloor}}\frac{S_{k}(t)}{B_{k}(t)}$ for both sides  we get
(\ref{impid2}).
\end{proof}

We are ready now to prove the following theorem.
\begin{thm}
Let $k$ be a positive integer. Then the following identity holds:
\begin{equation*}
\sum_{i=k2^{n}}^{k2^{n+1}}\frac{1}{B_{i}(t)B_{i+1}(t)}=\frac{2-t}{t^{n+\lfloor
\log
k\rfloor+1}B_{k}(t)}S_{k}(t)+\frac{1}{t^{n+1}B_{k}(t)}\left(\frac{1}{B_{k2^{n+1}+1}(t)}+1\right),
\end{equation*}
where $S_{1}(t)=S_{2}(t)=0$ and for $k\geq 2$ we have
\begin{equation*}
S_{2k}(t)=tS_{k}(t),\quad
S_{2k+1}(t)=S_{k}(t)+S_{k+1}(t)+t^{\lfloor\log k \rfloor}.
\end{equation*}
\end{thm}
\begin{proof}
Let $P_{k,n}$ (respectively $Q_{k,n}$) denote the left hand side
(respectively the right hand side) of the identity which we want to
prove. We prove the desired identity in two steps. First we prove
that $P_{k,n}$ and $Q_{k,n}$ satisfy the same linear recurrence
relation (with respect to $n$). In Step 2 we use similar reasoning
in order to prove that $P_{k,0}=Q_{k,0}$ for $k\in\N$. These two
facts tied together give the result.

Step 1. We find a recurrence relation satisfied by $P_{k,n}$. In
order to shorten the notation we put $B_{i}:=B_{i}(t)$. We have the
following chain of equalities
\begin{align*}
P&_{k,n+1}=\\
 &\sum_{i=k2^{n+1}}^{k2^{n+2}}\frac{1}{B_{i}B_{i+1}}=\sum_{i=k2^{n}}^{k2^{n+1}}\frac{1}{B_{2i}B_{2i+1}}+\sum_{i=k2^{n}}^{k2^{n+1}}\frac{1}{B_{2i+1}B_{2i+2}}-\frac{1}{B_{k2^{n+2}+1}B_{k2^{n+2}+2}}\\
 &=\sum_{i=k2^{n}}^{k2^{n+1}}\left(\frac{1}{tB_{i}(B_{i}+B_{i+1})}+\frac{1}{tB_{i+1}(B_{i}+B_{i+1})}\right)-\frac{1}{B_{k2^{n+2}+1}B_{k2^{n+2}+2}}\\
 &=\frac{1}{t}\sum_{i=k2^{n}}^{k2^{n+1}}\frac{B_{i}+B_{i+1}}{B_{i}B_{i+1}(B_{i}+B_{i+1})}-\frac{1}{B_{k2^{n+2}+1}B_{k2^{n+2}+2}}\\
 &=\frac{1}{t}P_{k,n}-\frac{1}{B_{k2^{n+2}+1}B_{k2^{n+2}+2}}=\frac{1}{t}P_{k,n}-\frac{1}{tB_{k2^{n+1}+1}B_{k2^{n+2}+1}}
\end{align*}
We thus find that $P_{k,n+1}=t^{-1}P_{k,n}-1/tB_{k2^{n+1}+1}B_{k2^{n+2}+1}$. We show that $Q_{k,n}$ satisfies exactly the same recurrence relation. First of all we note that from the definition of $Q_{k,n}$ we have
\begin{equation*}
 Q_{k,n+1}=\frac{1}{t}Q_{k,n}+\frac{1}{t^{n+2}B_{k}}\left(\frac{1}{B_{k2^{n+2}+1}}-\frac{1}{B_{k2^{n+1}+1}}\right).
\end{equation*}
Using the recurrence relation satisfied by $B_{n}$ we easily get that the expression in the bracket is equal to
\begin{equation*}
 \frac{B_{k2^{n+1}+1}-B_{k2^{n+2}+1}}{B_{k2^{n+1}+1}B_{k2^{n+2}+1}}=\frac{B_{k2^{n+1}+1}-(B_{k2^{n+1}+1}+B_{k2^{n+1}})}{B_{k2^{n+1}+1}B_{k2^{n+2}+1}}=-\frac{t^{n+1}B_{k}}{B_{k2^{n+1}+1}B_{k2^{n+2}+1}}.
\end{equation*}
Putting this into the expression for $Q_{k,n+1}$ we easily get that $Q_{k,n}$ satisfies the same recurrence relation as $P_{k,n}$.

Step 2. We prove that $P_{k,0}=Q_{k,0}$. In order to do this we use
similar reasoning as in Step 1. We find a recurrence relation for
$P_{k}:=P_{k,0}$ and show that $Q_{k}:=Q_{k,0}$ satisfies the same
relation. We start with finding a recurrence relation for $P_{k}$.
We note that $P_{2k,n}=P_{k,n+1}$ and thus $P_{2k}=P_{2k,0}=P_{k,1}$
and using the recurrence relation from Step 1 we easily deduce that
\begin{equation*}
P_{2k}=\frac{1}{t}P_{k}-\frac{1}{tB_{2k+1}B_{4k+1}}.
\end{equation*}
Using now the expression for $P_{2k}$ we get
\begin{align*}
P_{2k+1}&=P_{2k}+\frac{1}{tB_{2k+1}}\left(\frac{1}{B_{4k+1}}
+\frac{1}{B_{4k+3}}-\frac{1}{B_{k}}\right)\\
        &=\frac{1}{t}P_{k}-\frac{1}{tB_{2k+1}B_{4k+1}}+\frac{1}{tB_{2k+1}}\left(\frac{1}{B_{4k+1}}
+\frac{1}{B_{4k+3}}-\frac{1}{B_{k}}\right).
\end{align*}
Using now the recurrence relation for defining the sequence
$\{B_{n}\}_{n=0}^{\infty}$ and performing simple but tiresome
calculations we arrive at the expression
\begin{equation*}
P_{2k+1}=\frac{1}{t}P_{k}-\frac{(t+1)B_{k+1}}{tB_{k}B_{2k+1}B_{4k+3}}.
\end{equation*}

We consider now the sequence $Q_{k}$. Using exactly the same
reasoning as in the case of $P_{2k}$ we deduce that
\begin{equation*}
Q_{2k}=\frac{1}{t}Q_{k}-\frac{1}{tB_{2k+1}B_{4k+1}}.
\end{equation*}

We prove that $Q_{2k+1}$ satisfies the same relation as $P_{2k+1}$.
Using now the result from Corollary \ref{impid2} and the identity
$\lfloor \log_{2}(k+1) \rfloor=\lfloor \log_{2}k \rfloor +1$ we get
that
\begin{align*}
Q_{2k+1}&=\frac{2-t}{t^{\lfloor
\log_{2}(2k+1)\rfloor+1}}\frac{S_{2k+1}}{B_{2k+1}}+\frac{1}{tB_{2k+1}}\left(\frac{1}{B_{4k+3}}+1\right)\\
        &=\frac{1}{t}\left(\frac{2-t}{t^{\lfloor
\log_{2}k\rfloor+1}}\frac{S_{k}}{B_{k}}+\frac{B_{k+1}-(t-1)(B_{k}+1)}{tB_{k}B_{2k+1}}\right)+\frac{1}{tB_{2k+1}}\left(\frac{1}{B_{4k+3}}+1\right)
\end{align*}
We note that
\begin{equation*}
\frac{1}{t}\frac{2-t}{t^{\lfloor
\log_{2}k\rfloor+1}}\frac{S_{k}}{B_{k}}=\frac{1}{t}Q_{k}-\frac{1}{t^2B_{k}}\left(\frac{1}{B_{2k+1}}+1\right).
\end{equation*}
Using this expression we get that
\begin{align*}
&Q_{2k+1}\\
&=\frac{1}{t}Q_{k}-\frac{1}{t^2B_{k}}\left(\frac{1}{B_{2k+1}}+1\right)+\frac{B_{k+1}-(t-1)(B_{k}+1)}{t^2B_{k}B_{2k+1}}+\frac{1}{tB_{2k+1}}\left(\frac{1}{B_{4k+3}}+1\right)
\end{align*}
We are thus left with the simplification of the complicated
expression in the above identity. One can easily see that the common
denominator, say $D$, of the sum of fractions which arises on the
left side in the expression for $Q_{2k+1}$ is clearly
$D=t^2B_{k}B_{2k+1}B_{4k+3}$. The numerator, say $N$, after
simplifications is equal to
\begin{equation*}
N=tB_{k}-(t-B_{k}-B_{k+1}+B_{2k+1})B_{4k+3}=t(B_{k}-B_{4k+3})=-t(t+1)B_{k+1},
\end{equation*}
where in the last equality we use the identity
$B_{4k+3}=B_{k}+(t+1)B_{k+1}$. We thus obtain the relation
\begin{equation*}
Q_{2k+1}=\frac{1}{t}Q_{k}-\frac{(t+1)B_{k+1}}{tB_{k}B_{2k+1}B_{4k+3}},
\end{equation*}
which is exactly the same relation satisfied by $P_{k}$. Gathering
the information we have obtained we see that $P_{k}$ and $Q_{k}$
satisfy the same recurrence relation.

In order to finish the proof of our theorem we note that
$P_{1,0}=(t+2)/t(t+1)=Q_{1,0}$ and because $P_{k}$ and $Q_{k}$
satisfy the same recurrence relation with the same initial condition
we deduce that $P_{k,0}=P_{k}=Q_{k}=Q_{k,0}$ for all $k\in\N_{+}$.
Using now the information from Step 1 we get that $P_{k,n}=Q_{k,n}$
for all $k\in\N_{+}$ and $n\in\N$. This observation finishes the
proof of our theorem.
\end{proof}

\section{Some additional observations, open questions and conjectures}\label{section6}

We start this section with a very simple result on the partial sums
of the $\pm1$ sequence $\{(-1)^{e(i)}\}_{i=1}^{\infty}$.

\begin{thm}
Let $e(n)=\op{deg}B_{n}(t)$ and let us consider the sequence
$S(n)=\sum_{i=1}^{n}(-1)^{e(i)}$. Then we have
\begin{equation*}
\liminf_{n\rightarrow+\infty}S(n)=-\infty,\quad\quad
\limsup_{n\rightarrow+\infty}S(n)=\infty,
\end{equation*}
\end{thm}
\begin{proof}
In order to get the demanded equalities we use the value of
$s_{n}:=S(2^{n})$ which is just $H_{n}(-1)$, where $H_{n}$ is the
polynomial considered in Section \ref{section2}. We have
$s_{0}=1,\;s_{1}=0$ and for $n\geq 2$ we see that the sequence
$\{s_{n}\}_{n=0}^{\infty}$ satisfies the recurrence relation
$s_{n+2}=-s_{n+1}-2s_{n}+1+(-1)^{n}$. Although we have obtained a
closed expression for $H_{n}(x)$ in Corollary \ref{genforH(n)} it is
of little use in our situation. Instead of using it we just solve
the recurrence relation for $s_{n}$. We have
\begin{equation*}
s_{n}=\frac{1}{4}+\frac{(-1)^{n}}{2}+\epsilon_{1}p_{1}^{n}+\epsilon_{2}p_{2}^{n},
\end{equation*}
where $p_{1}=(-1+\sqrt{-7})/2,\;p_{2}=-(1+\sqrt{-7})/2$ and
$\epsilon_{1}=(7-3\sqrt{-7})/56,\;\epsilon_{2}=(7+3\sqrt{-7})/56$.
We introduce a sequence $\{t_{n}\}_{n=0}^{\infty}\subset\Z$ defined
in the following way
\begin{equation*}
t_{n}:=4\left(s_{n}-\frac{1}{4}-\frac{(-1)^{n}}{2}\right).
\end{equation*}
The sequence $t_{n}$ starts as follows
\begin{equation*}
1, 1, -3, 1, 5, -7, -3, 17, -11, -23, 45, 1, -91, 89, 93,
-271,\ldots .
\end{equation*}
One can easily check that $t_{n+2}=-t_{n+1}-2t_{n}$. Because
$|p_{1}|=|p_{2}|=\sqrt{2}>1$, thus from the theory of linear
difference equations we know that
$\lim_{n\rightarrow+\infty}|t_{n}|=+\infty$ and clearly the same
property holds for the sequence $s_{n}$. From the shape of the
recurrence relation for $t_{n}$ we deduce that $t_{n}$ changes sign
infinitely often (in fact there is no $n\in\N$ such that $t_{n},
t_{n+1},t_{n+2}$ are of the same sign). From the equality
$\lim_{n\rightarrow+\infty}|t_{n}|=+\infty$ and the mentioned
property of signs we immediately get that
\begin{equation*}
\liminf_{n\rightarrow+\infty}t_{n}=-\infty,\quad\quad
\limsup_{n\rightarrow+\infty}t_{n}=\infty,
\end{equation*}
and clearly the same property holds for $s_{n}$.
\end{proof}

In the next result we find the functional equation satisfied by the
ordinary generating function of the sequence
$\{e(n)\}_{n=1}^{\infty}$.

\begin{thm}\label{functeqfore}
Let $\cal{E}_{1}(x)\in\Z[[x]]$ be an ordinary generating function
for the sequence $\{e(n)\}$, i. e.
$\cal{E}_{1}(x)=\sum_{n=1}^{\infty}e(n)x^{n}$. Then $\cal{E}_{1}(x)$
satisfies the following functional equation:
\begin{equation*}
\cal{E}_{1}(x)=\cal{E}_{1}(x^2)+\frac{x^2+1}{x}\cal{E}_{1}(x^4)+\frac{x^2}{1-x}.
\end{equation*}
\end{thm}
\begin{proof}
This is an easy consequence of the recurrence relation for the
sequence $e(n)$. Indeed, we have the following chain of equalities
\begin{align*}
\cal{E}_{1}(x)&=\sum_{n=1}^{\infty}e(n)x^{n}=\sum_{n=1}^{\infty}e(2n)x^{2n}+\sum_{n=0}^{\infty}e(4n+1)x^{4n+1}+\sum_{n=0}^{\infty}e(4n+3)x^{4n+3}\\
              &=\sum_{n=1}^{\infty}(e(n)+1)x^{2n}+x\sum_{n=1}^{\infty}(e(n)+1)x^{4n}+\frac{1}{x}\sum_{n=1}^{\infty}(e(n)+1)x^{4n}\\
              &=\cal{E}_{1}(x^2)+\frac{x^2+1}{x}\cal{E}_{1}(x^4)+\frac{x^2}{1-x}.
\end{align*}
Note that in the second last equality we use the recurrence relation
for $e(n)$. The result follows.
\end{proof}

Before we give the next result let us recall two results obtained in
\cite{Bec}.

\begin{prop}[Corollary 1 and Lemma 5 from \cite{Bec}]
(1) Let $\mathbb{F}$ be a field and let $f(x)\in\mathbb{F}[[x]]$ be
a power series satisfying the equation
\begin{equation*}
\sum_{i=0}^{m}a_{i}(x)f(x^{k^{i}})=0,
\end{equation*}
where $a_{0}(x),\ldots,a_{m}(x)\in\mathbb{F}(x)$ and $a_{0}(x)\equiv
1$. Suppose that there exists a rational function
$r(x)\in\mathbb{F}(x)\setminus\{0\}$ whose poles (in the algebraic
closure of $\mathbb{F}$) are either zero or roots of unity. If
\begin{equation*}
\frac{r(x^{k^{i}})}{r(x)}a_{i}(x)\in\mathbb{F}[x]
\end{equation*}
for $i=1,2,\ldots,m$, then $f(x)$ is a $k$-regular power series.

(2) If $f\in\bar{\Q}[[x]]$ be a $k$-regular power series. Then it is
either a rational function or it is transcendental over $\Q(x)$.
\end{prop}

We use the cited result in order to prove the following.

\begin{cor}\label{transcofgenfore}
The function $\cal{E}_{1}(x)$ is 2-regular power series and it is
transcendental over $\Q(x)$.
\end{cor}
\begin{proof}
First of all we rewrite the functional equation for $\cal{E}_{1}(x)$
in the following form
\begin{equation*}
\cal{E}_{1}(x)-\frac{x^2+x+1}{x^2}\cal{E}_{1}(x^2)-\frac{x^3-1}{x^2}\cal{E}_{1}(x^4)+\frac{(1+x)(1+x^4)}{x^4}\cal{E}_{1}(x^8)=0.
\end{equation*}
The displayed identity follows from the functional equation for
$\cal{E}_{1}$ with substitution $x^2$ instead of $x$ and uses the
fact that $\frac{x^{4}}{1-x^2}=\frac{x^{2}}{1+x}\frac{x^2}{1-x}$.
From the result of Becker we deduce that $\cal{E}_{1}$ is 2-regular
by taking $r(x)=x^{4}$.

In order to prove transcendence of $\cal{E}_{1}(x)$ it is enough to
show that $\cal{E}_{1}(x)$ is not a rational function. So let us
suppose that $\cal{E}_{1}(x)=p(x)/q(x)$ with $p,q\in\Z[x]$ and
$\gcd(p(x),q(x))=1$. From this we deduce that the power series
$\cal{E}(x):=\cal{E}_{1}(x)\pmod{2}$ is rational over
$\mathbb{F}_{2}$ (a field with two elements). But then
$\cal{E}(x^{2^{i}})=\cal{E}_{1}(x)^{2^{i}}$ and thus we get that
$\cal{E}(x)$ satisfies an algebraic equation
\begin{equation}\label{eqforT}
F(x,T)=(1-x)(1+x^2)T^4+x(1-x)T^2-x(1-x)T+x^2=0.
\end{equation}
Because $\cal{E}$ is rational we know that the equation $F(x,T)=0$
has a rational root. But one can easily check that there is no
rational function over $\mathbb{F}_{2}$ for which (\ref{eqforT})
holds. This contradiction finishes the proof of the transcendence of
$\cal{E}_{1}(x)$.
\end{proof}

Using similar reasoning as in the proof of Theorem \ref{sumofi} and
Proposition \ref{functeqfore} one can easily deduce the following.

\begin{thm}
Let $k$ be a nonnegative integer and let us define the function
$\cal{E}_{k}(x)=\sum_{n=1}^{\infty}e(n)^{k}x^{n}$. Then
$\cal{E}_{0}(x)=\frac{x}{1-x}$ and for $k\geq 1$ the function
$\cal{E}_{k}(x)$ satisfies the functional equation
\begin{equation*}
\cal{E}_{k}(x)-\cal{E}_{k}(x^2)-\frac{x^2+1}{x}\cal{E}_{k}(x^4)=\sum_{j=0}^{k-1}C(k,j)\left(\cal{E}_{j}(x^2)+\frac{x^2+1}{x}\cal{E}_{j}(x^4)\right).
\end{equation*}
\end{thm}

We know that $\cal{E}_{0}(x)$ is rational and from Corollary
\ref{transcofgenfore} we know that the function $\cal{E}_{1}(x)$ is
transcendental. It is an interesting question whether the function
$\cal{E}_{k}(x)$ for $k\geq 2$ is transcendental too. We believe
that this is the case and it leads us to the following.

\begin{conj}
The function $\cal{E}_{k}(x)$ is transcendental for $k\geq 2$.
\end{conj}

Finally, we state the following conjecture which appeared during our
investigations of the sequence of maximal coefficients of the Stern
polynomial $B_{n}(t)$.

\begin{conj}
Let $B_{n}(t)=\sum_{i=0}^{e(n)}a_{i,e(n)}x^{i}$ and let us define
$\cal{M}(n)=\op{max}\{a_{i,e(n)}:\;i=1,2,\ldots, e(n)\}$. Then the
following equality holds
\begin{equation*}
\op{max}\{\cal{M}(m):\;m\in
[2^{n-1},2^{n}]\}=\op{max}\left\{C(n,0),C(n-1,1),\ldots,C(n-k,k)\right\},
\end{equation*}
where $k$ is equal to $n/2$ if $n$ is even and $(n-1)/2$ for $n$
odd.
\end{conj}

\bigskip

\noindent Jagiellonian University, Institute of Mathematics,
{\L}ojasiewicza 6, 30-348 Krak\'ow, Poland; e-mail:\; {\tt
maciej.ulas@uj.edu.pl}

 \end{document}